\documentclass{article}

\usepackage{amsmath,amssymb,amsfonts,amsthm}
\usepackage{accents}
\usepackage{fullpage}
\usepackage{color}
\usepackage{tikz}
\usetikzlibrary{decorations.pathreplacing}
\usepackage{graphicx}

\newcommand{\bk}{\bold{m}}
\newtheorem{theorem}{Theorem}[section]
\newtheorem{prop}[theorem]{Proposition}
\newtheorem{lemma}[theorem]{Lemma}

\theoremstyle{remark}

\newtheorem{example}{Example}
\newtheorem{definition}[theorem]{Definition}

\begin{document}

\title{Stochastic Baxterisation of a fused Hecke algebra}

\author{Jeffrey Kuan}

\date{}

\maketitle

\abstract{Baxterisation is a procedure which constructs solutions of the Yang--Baxter equation from algebra representations. A recent paper \cite{crampebax} provides Baxterisation formulas for a fused Hecke algebra. In this paper, we provide a stochastic version of Baxterisation for the fused Hecke algebra, which yields stochastic solutions to the Yang--Baxter equation. The coefficients in the Baxterisation formula previously appeared as the $q$--Hahn weights. This results in new formulas for the weights of the stochastic higher--spin vertex model.  

\section{Introduction}
The Yang--Baxter equation is a well--known equation arising from theoretical physics. Its importance arises from its use in the quantum inverse scattering method, or algebraic Bethe Ansatz, which goes back at least forty years \cite{STF79}. In more recent years, \textit{stochastic} solutions to the Yang--Baxter equation have been studied in the field of integrable probability (e.g. \cite{CorPetCMP,Bor16,BP,BWColor}). In the probabilistic context, the Yang--Baxter equation leads to exact formulas for expectations of certain observables.

Generally speaking, it is difficult to find stochastic solutions to the Yang--Baxter equation. One fruitful method has been to use the quasi--triangular structure of Drinfeld--Jimbo quantum groups (e.g. \cite{KMMO}). In this paper, we propose another method to produce stochastic solutions to the Yang--Baxter equation, called \textit{stochastic Baxterisation}. The usual Baxterisation starts with a solution of the braid equation, given by some algebra representation, and constructs a solution the Yang--Baxter equation with spectral parameters. Baxterisation goes back to \cite{Jones:1990hq}. In the stochastic variant, the coefficients of the algebra elements are stochastic weights, and taking a stochastic representation yields matrix--valued solutions to the Yang--Baxter equation. We carry out the calculation for the fused Hecke algebra of \cite{crampe2020fused,crampebax}. This yields new formulas for the stochastic higher--spin vertex model. The method is expected to work for other algebras (e.g. \cite{CGX,ZGB,Li93,BM01,Arn,Kulish:2009cx,doi:10.1063/1.4905893,Crampe:2016nek,VanThesis,Crampe:2019cig})

Section \ref{BG} contains the relevant background, and section \ref{SV} contains the new results. Most of the background material is taken from \cite{crampe2020fused,crampebax}.
 
\section{Background}\label{BG}
Recall that the usual Hecke algebra $H_n(q)$ is generated by elements $\sigma_1, \ldots, \sigma_{n-1}$ with defining relations:
$$
\begin{array}{ll}\sigma_{i}^{2}=\left(q-q^{-1}\right) \sigma_{i}+1 & \text { for } i \in\{1, \ldots, n-1\} \\ \sigma_{i} \sigma_{i+1} \sigma_{i}=\sigma_{i+1} \sigma_{i} \sigma_{i+1} & \text { for } i \in\{1, \ldots, n-2\} \\ \sigma_{i} \sigma_{j}=\sigma_{j} \sigma_{i} & \text { for } i, j \in\{1, \ldots, n-1\} \text { such that }|i-j|>1\end{array}$$
It is well--known, and straightforward to check, that the elements 
\begin{equation}\label{eq:Rhecke}
\check{R}_{i}(u):=\check{R}_{i}^{(1)}(u)=\sigma_{i}-\frac{q-q^{-1}}{1-u}
\end{equation}
satisfy the Yang--Baxter equation
$$
\check{R}_{i}^{(1)}(u) \check{R}_{i+1}^{(1)}(u v) \check{R}_{i}^{(1)}(v)=\check{R}_{i+1}^{(1)}(v) \check{R}_{i}^{(1)}(u v) \check{R}_{i+1}^{(1)}(v)
$$

Before defining the fused Hecke algebra from \cite{crampe2020fused}, we first define fused permutations in terms of diagrams. Place two rows of $L$ ellipses, one on top of the other, and connect top ellipses with bottom ellipses with edges. Fix a sequence of positive integers $\mathbf{m}=(m_1,\ldots,m_L)$. For each $x \in \{1,\ldots,L\}$, exactly $m_x$ edges start from the $x$--th top ellipses and exactly $m_x$ edges arrive at the $x$--th bottom ellipse. For each $x \in \{1,\ldots,L\}$, let $I_x$ denote the multi-set indicating the bottom ellipses reached by the $m_x$ edges starting from the $x$--th top ellipse. Two diagrams are equivalent if their multisets $(I_1,\ldots,I_L)$ coincide. A \textit{fused permutation} is an equivalence class of diagrams. Below is an example of a fused permutation with the corresponding multi--set.

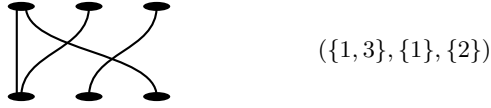
\begin{center}
\begin{tikzpicture}[scale=0.3]

\fill (21,2) ellipse (0.6cm and 0.2cm);\fill (21,-2) ellipse (0.6cm and 0.2cm);
\draw[thick] (20.8,2) -- (20.8,-2);\draw[thick] (21.2,2)..controls +(0,-2) and +(0,+2) .. (27,-2);  
\fill (24,2) ellipse (0.6cm and 0.2cm);\fill (24,-2) ellipse (0.6cm and 0.2cm);
\draw[thick] (24,2)..controls +(0,-2) and +(0,+2) .. (21,-2); 
\fill (27,2) ellipse (0.6cm and 0.2cm);\fill (27,-2) ellipse (0.6cm and 0.2cm);
\draw[thick] (27,2)..controls +(0,-2) and +(0,+2)..(24,-2);
\node at (38,0) {$(\{1,3\},\{1\},\{2\})$};

\end{tikzpicture}
\end{center}
If all $m_a=1$, the fused permutation is just a usual permutation in the symmetric group $S_L$.

Fused permutations can be multiplied by concatenating two diagrams, removing the middle ellipses, and normalizing. Because the precise form of the multiplication will not be needed here, we refer to \cite{crampe2020fused} for details. 

\begin{definition}
The $\mathbb{C}$-vector space $H_{\bk,L}(q)$ is the quotient of the vector space with basis indexed by fused braids by the following relations:
 \begin{itemize}
   \item[(i)] The Hecke relation:  
   \begin{center}
 \begin{tikzpicture}[scale=0.25]
\draw[thick] (0,2)..controls +(0,-2) and +(0,+2) .. (4,-2);
\fill[white] (2,0) circle (0.4);
\draw[thick] (4,2)..controls +(0,-2) and +(0,+2) .. (0,-2);
\node at (6,0) {$=$};
\draw[thick] (12,2)..controls +(0,-2) and +(0,+2) .. (8,-2);
\fill[white] (10,0) circle (0.4);
\draw[thick] (8,2)..controls +(0,-2) and +(0,+2) .. (12,-2);
\node at (17,0) {$-\,(q-q^{-1})$};
\draw[thick] (21,2) -- (21,-2);\draw[thick] (25,2) -- (25,-2);
\end{tikzpicture}
\end{center}
  \item[(ii)]  The idempotent relations: for top ellipses,
 \begin{center}
 \begin{tikzpicture}[scale=0.4]
\fill (2,2) ellipse (0.8cm and 0.2cm);
\draw[thick] (2.2,2)..controls +(0,-1.5) and +(1,1) .. (1.2,0);
\fill[white] (2,0.7) circle (0.2);
\draw[thick] (1.8,2)..controls +(0,-1.5) and +(-1,1) .. (2.8,0);
\node at (4.5,1) {$=$};
\node at (7,1) {$q$};
\fill (9,2) ellipse (0.8cm and 0.2cm);
\draw[thick] (8.8,2)..controls +(0,-1.5) and +(0.5,0.5) .. (8.2,0);
\draw[thick] (9.2,2)..controls +(0,-1.5) and +(-0.5,0.5) .. (9.8,0);

\node at (13,1) {and};

\fill (18,2) ellipse (0.8cm and 0.2cm);
\draw[thick] (17.8,2)..controls +(0,-1.5) and +(-1,1) .. (18.8,0);
\fill[white] (18,0.7) circle (0.2);
\draw[thick] (18.2,2)..controls +(0,-1.5) and +(1,1) .. (17.2,0);
\node at (20.5,1) {$= $};
\node at (23,1) {$q^{-1}$};
\fill (25,2) ellipse (0.8cm and 0.2cm);
\draw[thick] (24.8,2)..controls +(0,-1.5) and +(0.5,0.5) .. (24.2,0);
\draw[thick] (25.2,2)..controls +(0,-1.5) and +(-0.5,0.5) .. (25.8,0);
\end{tikzpicture}
\end{center}
and for bottom ellipses,
\begin{center}
 \begin{tikzpicture}[scale=0.4]
\fill (2,0) ellipse (0.8cm and 0.2cm);
\draw[thick] (1.8,0)..controls +(0,1.5) and +(-1,-1) .. (2.8,2);
\fill[white] (2,1.25) circle (0.2);
\draw[thick] (2.2,0)..controls +(0,1.5) and +(1,-1) .. (1.2,2);
\node at (4.5,1) {$=$};
\node at (7,1) {$q$};
\fill (9,0) ellipse (0.8cm and 0.2cm);
\draw[thick] (9.2,0)..controls +(0,1.5) and +(-0.5,-0.5) .. (9.8,2);
\draw[thick] (8.8,0)..controls +(0,1.5) and +(0.5,-0.5) .. (8.2,2);

\node at (13,1) {and};

\fill (18,0) ellipse (0.8cm and 0.2cm);
\draw[thick] (18.2,0)..controls +(0,1.5) and +(1,-1) .. (17.2,2);
\fill[white] (18,1.25) circle (0.2);
\draw[thick] (17.8,0)..controls +(0,1.5) and +(-1,-1) .. (18.8,2);
\node at (20.5,1) {$=$};
\node at (23,1) {$q^{-1}$};
\fill (25,0) ellipse (0.8cm and 0.2cm);
\draw[thick] (25.2,0)..controls +(0,1.5) and +(-0.5,-0.5) .. (25.8,2);
\draw[thick] (24.8,0)..controls +(0,1.5) and +(0.5,-0.5) .. (24.2,2);
\end{tikzpicture}
\end{center}
 \end{itemize}
 \end{definition}

\begin{definition}
Let $i \in \{1,\ldots,n-1\}$ and $ 0 \leq p \leq k$. The partial elementary braiding $\Sigma_i^{(k;p)} \in H_{k,n}(q)$ corresponds to the diagram in which the $p$ rightmost strands starting from ellipse $i$ pass over the $p$ leftmost strands starting from ellipse $i+1$. All other strands are vertical.
\end{definition}

They satisfy:
$$\begin{aligned} \Sigma_{i}^{(k ; k)} \Sigma_{i+1}^{(k ; k)} \Sigma_{i}^{(k ; k)} &=\Sigma_{i+1}^{(k ; k)} \Sigma_{i}^{(k ; k)} \Sigma_{i+1}^{(k ; k)} \\ \Sigma_{i}^{(k ; k)} \Sigma_{j}^{(k ; k)} &=\Sigma_{j}^{(k ; k)} \Sigma_{i}^{(k ; k)} \end{aligned} \quad \text { if }|i-j|>1
$$

Some examples:

\begin{center}
 \begin{tikzpicture}[scale=0.23]
\node at (-3,0) {$\Sigma_i^{(2;0)}:=$};
\node at (2,3) {$1$};\fill (2,2) ellipse (0.6cm and 0.2cm);\fill (2,-2) ellipse (0.6cm and 0.2cm);
\draw[thick] (1.8,2) -- (1.8,-2);\draw[thick] (2.2,2) -- (2.2,-2);
\node at (4,0) {$\dots$};
\draw[thick] (5.8,2) -- (5.8,-2);\draw[thick] (6.2,2) -- (6.2,-2);
\node at (6,3) {$i-1$};\fill (6,2) ellipse (0.6cm and 0.2cm);\fill (6,-2) ellipse (0.6cm and 0.2cm);
\node at (10,3) {$i$};\fill (10,2) ellipse (0.6cm and 0.2cm);\fill (10,-2) ellipse (0.6cm and 0.2cm);
\node at (14,3) {$i+1$};\fill (14,2) ellipse (0.6cm and 0.2cm);\fill (14,-2) ellipse (0.6cm and 0.2cm);

\draw[thick] (9.8,2)..controls +(0,-2) and +(0,+2) .. (9.8,-2);
\draw[thick] (10.2,2)..controls +(0,-2) and +(0,+2) .. (10.2,-2);
\fill[white] (12,0) circle (0.5);
\draw[thick] (14.2,2)..controls +(0,-2) and +(0,+2) .. (14.2,-2);
\draw[thick] (13.8,2)..controls +(0,-2) and +(0,+2) .. (13.8,-2);

\draw[thick] (17.8,2) -- (17.8,-2);\draw[thick] (18.2,2) -- (18.2,-2);
\node at (18,3) {$i+2$};\fill (18,2) ellipse (0.6cm and 0.2cm);\fill (18,-2) ellipse (0.6cm and 0.2cm);
\node at (20,0) {$\dots$};
\draw[thick] (21.8,2) -- (21.8,-2);\draw[thick] (22.2,2) -- (22.2,-2);\fill (22,2) ellipse (0.6cm and 0.2cm);\fill (22,-2) ellipse (0.6cm and 0.2cm);
\node at (22,3) {$n$};
\end{tikzpicture}
\end{center}

\begin{center}
 \begin{tikzpicture}[scale=0.23]
\node at (-3,0) {$\Sigma_i^{(2;1)}:=$};
\node at (2,3) {$1$};\fill (2,2) ellipse (0.6cm and 0.2cm);\fill (2,-2) ellipse (0.6cm and 0.2cm);
\draw[thick] (1.8,2) -- (1.8,-2);\draw[thick] (2.2,2) -- (2.2,-2);
\node at (4,0) {$\dots$};
\draw[thick] (5.8,2) -- (5.8,-2);\draw[thick] (6.2,2) -- (6.2,-2);
\node at (6,3) {$i-1$};\fill (6,2) ellipse (0.6cm and 0.2cm);\fill (6,-2) ellipse (0.6cm and 0.2cm);
\node at (10,3) {$i$};\fill (10,2) ellipse (0.6cm and 0.2cm);\fill (10,-2) ellipse (0.6cm and 0.2cm);
\node at (14,3) {$i+1$};\fill (14,2) ellipse (0.6cm and 0.2cm);\fill (14,-2) ellipse (0.6cm and 0.2cm);

\draw[thick] (9.8,2) -- (9.8,-2);
\draw[thick] (13.8,2)..controls +(0,-2) and +(0,+2) .. (10.2,-2);
\fill[white] (12,0) circle (0.4);
\draw[thick] (10.2,2)..controls +(0,-2) and +(0,+2) .. (13.8,-2);
\draw[thick] (14.2,2) -- (14.2,-2);

\draw[thick] (17.8,2) -- (17.8,-2);\draw[thick] (18.2,2) -- (18.2,-2);
\node at (18,3) {$i+2$};\fill (18,2) ellipse (0.6cm and 0.2cm);\fill (18,-2) ellipse (0.6cm and 0.2cm);
\node at (20,0) {$\dots$};
\draw[thick] (21.8,2) -- (21.8,-2);\draw[thick] (22.2,2) -- (22.2,-2);\fill (22,2) ellipse (0.6cm and 0.2cm);\fill (22,-2) ellipse (0.6cm and 0.2cm);
\node at (22,3) {$n$};
\end{tikzpicture}
\end{center}

\begin{center}
 \begin{tikzpicture}[scale=0.23]
\node at (-3,0) {$\Sigma_i^{(2;2)}:=$};
\node at (2,3) {$1$};\fill (2,2) ellipse (0.6cm and 0.2cm);\fill (2,-2) ellipse (0.6cm and 0.2cm);
\draw[thick] (1.8,2) -- (1.8,-2);\draw[thick] (2.2,2) -- (2.2,-2);
\node at (4,0) {$\dots$};
\draw[thick] (5.8,2) -- (5.8,-2);\draw[thick] (6.2,2) -- (6.2,-2);
\node at (6,3) {$i-1$};\fill (6,2) ellipse (0.6cm and 0.2cm);\fill (6,-2) ellipse (0.6cm and 0.2cm);
\node at (10,3) {$i$};\fill (10,2) ellipse (0.6cm and 0.2cm);\fill (10,-2) ellipse (0.6cm and 0.2cm);
\node at (14,3) {$i+1$};\fill (14,2) ellipse (0.6cm and 0.2cm);\fill (14,-2) ellipse (0.6cm and 0.2cm);

\draw[thick] (13.8,2)..controls +(0,-2) and +(0,+2) .. (9.8,-2);
\draw[thick] (14.2,2)..controls +(0,-2) and +(0,+2) .. (10.2,-2);
\fill[white] (12,0) circle (0.5);
\draw[thick] (10.2,2)..controls +(0,-2) and +(0,+2) .. (14.2,-2);
\draw[thick] (9.8,2)..controls +(0,-2) and +(0,+2) .. (13.8,-2);

\draw[thick] (17.8,2) -- (17.8,-2);\draw[thick] (18.2,2) -- (18.2,-2);
\node at (18,3) {$i+2$};\fill (18,2) ellipse (0.6cm and 0.2cm);\fill (18,-2) ellipse (0.6cm and 0.2cm);
\node at (20,0) {$\dots$};
\draw[thick] (21.8,2) -- (21.8,-2);\draw[thick] (22.2,2) -- (22.2,-2);\fill (22,2) ellipse (0.6cm and 0.2cm);\fill (22,-2) ellipse (0.6cm and 0.2cm);
\node at (22,3) {$n$};
\end{tikzpicture}
\end{center}

By Theorem 3.1 of \cite{crampebax}, the following elements of $H_{\mathbf{m},L}(q)$, for $1 \leq i \leq L-1$, 
\begin{equation}\label{Theorem}
\check{R}_{i}^{(k)}(u)=\sum_{p=0}^{k}(-q)^{k-p}\left[\begin{array}{c}k \\ p\end{array}\right]_{q}^{2} \frac{\left(q^{-2} ; q^{-2}\right)_{k-p}}{\left(u q^{-2 p} ; q^{-2}\right)_{k-p}} \Sigma_{i}^{(k ; p)}
\end{equation}
satisfy the braided Yang--Baxter equations, for $i \leq i \leq L-2$,
$$
\check{R}_{i}^{(k)}(u) \check{R}_{i+1}^{(k)}(u v) \check{R}_{i}^{(k)}(v)=\check{R}_{i+1}^{(k)}(v) \check{R}_{i}^{(k)}(u v) \check{R}_{i+1}^{(k)}(v).
$$
The proof of this result uses some additional definitions, which we describe below.

\paragraph{Projected Hecke algebra.}
In the Hecke algebra $H_{nk}(q)$, we consider the elements, for $1\leq i < j \leq nk-1$,
\begin{equation}\label{def-P}
S_{[i,j]}:=\frac{1}{[j-i+1 ]_q!} \ \prod_{i\leq a\leq j-1}^{\longrightarrow}\check R_a(q^{2(a-i+1)})\ \check R_{a-1}(q^{2(a-i)})\ \dots \check R_i(q^2) \ ,
\end{equation} 
where $\check R_a(u)$ is given by equation \eqref{eq:Rhecke} and $\displaystyle \prod_{i\leq a\leq j-1}^{\longrightarrow}$ means that the product is ordered from left to right when the index $a$ increases. This element is the image of $S_{[1,j-i+1]}$ through the embedding of $H_{j-i+1}(q)$ in $H_{nk}(q)$ given by $\sigma_a\mapsto\sigma_{a+i-1}$.

The element $S_{[1,j-i+1]}$ is the so-called $q$-symmetriser of the Hecke algebra $H_{j-i+1}(q)$. In fact, formula (\ref{def-P}) is the simplest case of the fusion formula for 
the Hecke algebra expressing a complete set of primitive idempotents as products of $R$-matrices \cite{IMO}. The $q$-symmetrisers are well-known 
objects (see \cite{IO09,Jimbo1986}). We recall below the main properties that we will use in the following.
A formula alternative to (\ref{def-P}) is the following sum:
\[S_{[i,j]}=\frac{q^{-\frac{(j-i+1)(j-i)}{2}}}{[j-i+1]_q!}\sum_w q^{\ell(w)}\sigma_w\ ,\]
where the sum runs over the set of permutations of the letters $\{i,i+1,\dots,j\}$, the elements $\sigma_w$ are the corresponding standard basis elements of $H_{nk}(q)$, and $\ell(w)$ is the number of crossings in the standard diagram $\sigma_w$ (equivalently, the length of $w$).
The element $S_{[i,j]}$ is a partial $q$-symmetriser satisfying in particular
\begin{equation}\label{rel-sym}
S_{[i,j]}^2=S_{[i,j]}\ \ \ \ \ \text{and}\ \ \ \ \ \ \sigma_a S_{[i,j]}=S_{[i,j]}\sigma_a=q S_{[i,j]}\,,\ a=i,\dots,j-1. 
\end{equation}
This implies that $S_{[i,j]}=S_{[i,j]}S_{[i',j']}$ if $i\leq i'<j'\leq j$. Finally, a recursion formula for these partial $q$-symmetrisers is:
\begin{equation}\label{rec-sym}
S_{[i,j+1]}=\frac{1}{[j-i+2]_q}\sum_{a=i}^{j+1}q^{i-a}\sigma_a\sigma_{a+1}\dots\sigma_j S_{[i,j]}\ .
\end{equation}

Inside the Hecke algebra $H_{nk}(q)$, we consider also the element
\begin{equation}\label{def-Pkn}
P^{(k)}:=S_{[1,k]}S_{[k+1,2k]}\dots  S_{[(n-1)k+1,nk]}\ .
\end{equation}
This is an idempotent (each factor is itself an idempotent and commutes with the others) and it allows us to construct $P^{(k)} H_{nk}(q)P^{(k)}$ which is an algebra with the unit $P^{(k)}$. Here we call this algebra the projected Hecke algebra. 
In \cite{crampe2020fused}, it is proved that the fused Hecke algebra $H_{k,n}(q)$ is isomorphic to the projected Hecke algebra $P^{(k)}H_{nk}(q) P^{(k)}$.

\paragraph{Elementary partial braidings in the projected Hecke algebra}
By using the isomorphism between $P^{(k)} H_{2k}(q)P^{(k)}$ and $H_{k,2}(q)$, the partial elementary braiding $\Sigma^{(k;p)} = \Sigma_1^{(k;p)}$ reads as follows
\begin{equation}\label{eq:Si}
 \Sigma^{(k;p)}= P^{(k)}  (\sigma_k \sigma_{k+1}\dots \sigma_{k+p-1}) (\sigma_{k-1} \sigma_{k}\dots \sigma_{k+p-2})  \dots  (\sigma_{k-p+1} \sigma_{k-p+1}\dots \sigma_{k})  P^{(k)} \ ,
\end{equation}
where $P^{(k)}=S_{[1,k]}S_{[k+1,2k]}$. Recall that $p\in\{0,\dots,k\}$ (when $p=0$, the element is simply $P^{(k)}$). The formula for $\Sigma^{(k;p)}$ is best visualised on the following diagram:
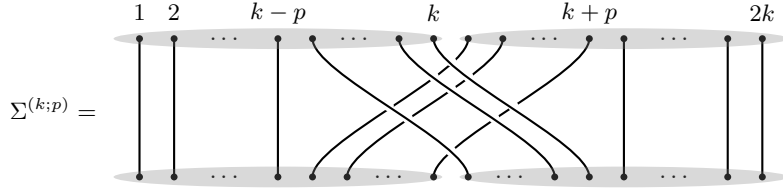
\begin{center}
 \begin{tikzpicture}[scale=0.23]
\node at (-3,0) {$\Sigma^{(k;p)}=$};
\node at (2,5.5) {$1$};\fill (2,4) circle (0.2cm);\fill (2,-4) circle (0.2cm);\draw[thick] (2,4) -- (2,-4);
\node at (4,5.5) {$2$};\fill (4,4) circle (0.2cm);\fill (4,-4) circle (0.2cm);\draw[thick] (4,4) -- (4,-4);
\node at (7,4) {$\dots$};\node at (7,-4) {$\dots$};
\node at (10,5.5) {$k-p$};\fill (10,4) circle (0.2cm);\fill (10,-4) circle (0.2cm);\draw[thick] (10,4) -- (10,-4);
\fill (12,4) circle (0.2cm);\fill (12,-4) circle (0.2cm);
\node at (14.5,4) {$\dots$};\node at (16.5,-4) {$\dots$};
\fill (17,4) circle (0.2cm);\fill (14,-4) circle (0.2cm);
\node at (19,5.5) {$k$};\fill (19,4) circle (0.2cm);\fill (19,-4) circle (0.2cm);

\fill (21,4) circle (0.2cm);\fill (21,-4) circle (0.2cm);
\fill (23,4) circle (0.2cm);\fill (26,-4) circle (0.2cm);
\node at (25.5,4) {$\dots$};\node at (23.5,-4) {$\dots$};
\node at (28,5.5) {$k+p$};\fill (28,4) circle (0.2cm);\fill (28,-4) circle (0.2cm);
\fill (30,4) circle (0.2cm);\fill (30,-4) circle (0.2cm);\draw[thick] (30,4) -- (30,-4);
\node at (33,4) {$\dots$};\node at (33,-4) {$\dots$};
\fill (36,4) circle (0.2cm);\fill (36,-4) circle (0.2cm);\draw[thick] (36,4) -- (36,-4);
\node at (38,5.5) {$2k$};\fill (38,4) circle (0.2cm);\fill (38,-4) circle (0.2cm);\draw[thick] (38,4) -- (38,-4);

\draw[thick] (21,4)..controls +(0,-2) and +(0,+2) .. (12,-4);
\draw[thick] (23,4)..controls +(0,-2) and +(0,+2) .. (14,-4);
\draw[thick] (28,4)..controls +(0,-2) and +(0,+2) .. (19,-4);

\fill[white] (20,2.7) circle (0.3);

\fill[white] (20,1.1) circle (0.3);
\fill[white] (20,-2.5) circle (0.3);
\fill[white] (21,1.8) circle (0.3);\fill[white] (19,1.8) circle (0.3);
\fill[white] (16.4,0.1) circle (0.3);\fill[white] (17.5,-0.6) circle (0.3);
\fill[white] (23.6,0.1) circle (0.3);\fill[white] (22.5,-0.6) circle (0.3);

\draw[thick] (19,4)..controls +(0,-2) and +(0,+2) .. (28,-4);
\draw[thick] (17,4)..controls +(0,-2) and +(0,+2) .. (26,-4);
\draw[thick] (12,4)..controls +(0,-2) and +(0,+2) .. (21,-4);

\fill[fill=gray,opacity=0.3] (10,4) ellipse (9.5cm and 0.6cm);
\fill[fill=gray,opacity=0.3] (10,-4) ellipse (9.5cm and 0.6cm);
\fill[fill=gray,opacity=0.3] (30,4) ellipse (9.5cm and 0.6cm);
\fill[fill=gray,opacity=0.3] (30,-4) ellipse (9.5cm and 0.6cm);

\end{tikzpicture}
\end{center}
The 4 shaded ellipses surrounding the dots $1,2, \dots k$ and $k+1,\dots,2k$ in the previous diagram represent the $q$-symmetrisers $S_{[1,k]}$ and $S_{[k+1,2k]}$.

More generally, let $l\geq k$ and define the following elements in $H_{k+\ell}(q)$, for $0\leq p\leq k$,
\begin{equation}\label{wut}
 \Sigma^{(k,\ell;p)}= P^{(k,\ell)}  (\sigma_k \sigma_{k+1}\dots \sigma_{\ell+p-1}) (\sigma_{k-1} \sigma_{k}\dots \sigma_{\ell+p-2})  \dots  (\sigma_{k-p+1} \sigma_{k-p+2}\dots \sigma_{\ell})  P^{(\ell,k)} \ ,
\end{equation}
where $P^{(k,\ell)}=S_{[1,k]}S_{[k+1,k+\ell]}$ and $P^{(\ell,k)}=S_{[1,\ell]}S_{[\ell+1,k+\ell]}$.
We get $\Sigma^{(k,k;p)}=\Sigma^{(k;p)}$.
These elements are visualised on the following diagrams:
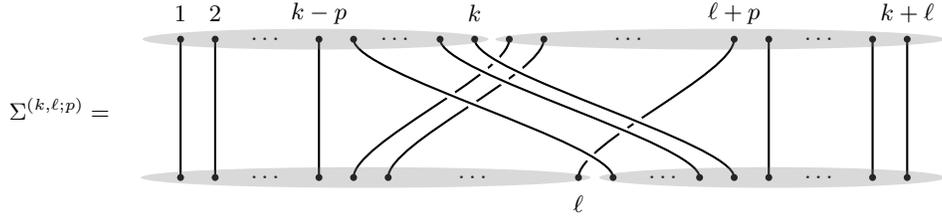
\begin{center}
 \begin{tikzpicture}[scale=0.23]
\node at (-5,0) {$\Sigma^{(k,\ell;p)}=$};
\node at (2,5.5) {$1$};\fill (2,4) circle (0.2cm);\fill (2,-4) circle (0.2cm);\draw[thick] (2,4) -- (2,-4);
\node at (4,5.5) {$2$};\fill (4,4) circle (0.2cm);\fill (4,-4) circle (0.2cm);\draw[thick] (4,4) -- (4,-4);
\node at (7,4) {$\dots$};\node at (7,-4) {$\dots$};
\node at (10,5.5) {$k-p$};\fill (10,4) circle (0.2cm);\fill (10,-4) circle (0.2cm);\draw[thick] (10,4) -- (10,-4);
\fill (12,4) circle (0.2cm);\fill (12,-4) circle (0.2cm);
\node at (14.5,4) {$\dots$};\node at (19,-4) {$\dots$};
\fill (17,4) circle (0.2cm);\fill (14,-4) circle (0.2cm);
\node at (19,5.5) {$k$};\fill (19,4) circle (0.2cm);\fill (32,-4) circle (0.2cm);

\fill (21,4) circle (0.2cm);\fill (25,-4) circle (0.2cm);
\fill (23,4) circle (0.2cm);\fill (27,-4) circle (0.2cm);
\node at (28,4) {$\dots$};\node at (30,-4) {$\dots$};
\node at (34,5.5) {$\ell+p$};\fill (34,4) circle (0.2cm);\fill (34,-4) circle (0.2cm);
\fill (36,4) circle (0.2cm);\fill (36,-4) circle (0.2cm);\draw[thick] (36,4) -- (36,-4);
\node at (39,4) {$\dots$};\node at (39,-4) {$\dots$};
\fill (42,4) circle (0.2cm);\fill (42,-4) circle (0.2cm);\draw[thick] (42,4) -- (42,-4);
\node at (44,5.5) {$k+\ell$};\fill (44,4) circle (0.2cm);\fill (44,-4) circle (0.2cm);\draw[thick] (44,4) -- (44,-4);

\node at (25,-5.5) {$\ell$};

\draw[thick] (21,4)..controls +(0,-2) and +(0,+2) .. (12,-4);
\draw[thick] (23,4)..controls +(0,-2) and +(0,+2) .. (14,-4);
\draw[thick] (34,4)..controls +(0,-2) and +(0,+2) .. (25,-4);

\fill[white] (20.2,2.8) circle (0.3);
\fill[white] (20.7,1.6) circle (0.3);
\fill[white] (25.7,-2.8) circle (0.3);
\fill[white] (21.5,2.2) circle (0.3);\fill[white] (19.5,2.2) circle (0.3);
\fill[white] (17.6,0.8) circle (0.3);\fill[white] (18.9,0.2) circle (0.3);
\fill[white] (28.4,-0.7) circle (0.3);\fill[white] (27.6,-1.3) circle (0.3);

\draw[thick] (19,4)..controls +(0,-2) and +(0,+2) .. (34,-4);
\draw[thick] (17,4)..controls +(0,-2) and +(0,+2) .. (32,-4);
\draw[thick] (12,4)..controls +(0,-2) and +(0,+2) .. (27,-4);

\fill[fill=gray,opacity=0.3] (9.8,4) ellipse (10cm and 0.6cm);
\fill[fill=gray,opacity=0.3] (12.7,-4) ellipse (13cm and 0.6cm);
\fill[fill=gray,opacity=0.3] (33.2,4) ellipse (13cm and 0.6cm);
\fill[fill=gray,opacity=0.3] (36.2,-4) ellipse (10cm and 0.6cm);

\end{tikzpicture}
\end{center}

\paragraph{Product of $R$-matrices.}

Let us now define, in $H_{k+\ell}(q) $,
\begin{equation}
 \check R^{(k,\ell)}(u) := P^{(k,\ell)} \prod_{1\leq a\leq k}^{\longleftarrow} \check R_a\left(uq^{2(1-a)}\right)\check R_{a+1}\left(uq^{2(2-a)}\right) \dots \check R_{a+\ell-1}\left(uq^{2(\ell-a)}\right)   P^{(\ell,k)}\ ,
 \label{eq:defkl}
\end{equation}
where  $\check R_a\left(u\right)$ is given by \eqref{eq:Rhecke}. The arrow means that the product is ordered from right to left when the index $a$ increases. It is a straightforward calculation, using repeatedly the braided Yang--Baxter equation satisfied by $\check R_i(u)$, to show that:
\begin{equation}\label{commPR}
P^{(k,\ell)}\check R^{(k,\ell)}(u)=\check R^{(k,\ell)}(u)P^{(\ell,k)}\ .
\end{equation}
As a consequence, only one idempotent (any one of the two) is actually needed in (\ref{eq:defkl}). 

\paragraph{Representations of the Hecke algebra} The Hecke algebra $H_n(q)$ has a representation on the vector space $(\mathbb{C}^{N})^{\otimes n}$. By abuse of notation, define the $R$--matrix $\check{R}$ 
from (12) of \cite{crampe2020fused} by 
$$
\check{R}\left(e_{i} \otimes e_{j}\right):=\left\{\begin{array}{ll}q e_{i} \otimes e_{j} & \text { if } i=j \\ e_{j} \otimes e_{i}+\left(q-q^{-1}\right) e_{i} \otimes e_{j} & \text { if } i<j, \quad \text { where } i, j=1, \ldots, N \\ e_{j} \otimes e_{i} & \text { if } i>j\end{array}\right.
$$
We write this as
$$
\check{R} = \sum_{i>j} E_{j,i} \otimes E_{i,j} + \sum_{i<j} E_{j,i} \otimes E_{i,j}+(q-q^{-1})\sum_{i<j} E_{i,i} \otimes E_{j,j}+ q\sum_i  E_{i,i} \otimes E_{i,i}
$$
Let $\check{R}_{i,i+1}$ denote the $R$--matrix acting on the $i,i+1$ tensor powers. Then $\sigma_i \mapsto R_{i,i+1}$ defines a representation of $H_n(q)$. 

More generally, as explained in section 5 of \cite{crampe2020fused}, this extends to a representation of $H_{\mathbf{k},n}$ of $L_{(k_1)}^N \otimes \cdots \otimes L{(k_n)}^N$, where $L_{(k)}$ is the irreducible representation of $U_q(\mathfrak{gl}_N)$ corresponding to the partition $(k)$, the one--line partition with $k$ boxes.

\section{Stochastic variants}\label{SV}
The main result of this paper is a stochastic version of the formula \eqref{Theorem}. To this end, we define stochastic versions of the Hecke algebra generators and symmetrizers. We use the accent $\circ$ to indicate a stochastic version.

If $\accentset{\circ}{\sigma}_i$ denotes $-q \sigma_i$, then $H_n(q)$ is generated by the elements $\accentset{\circ}{\sigma}_1, \ldots, \accentset{\circ}{\sigma}_{n-1}$ with relations
$$
\begin{array}{ll}
\accentset{\circ}{\sigma}_{i}^{2}=\left(1-q^2\right) \accentset{\circ}{\sigma}_i+q^2 & \text { for } i \in\{1, \ldots, n-1\} \\
 \accentset{\circ}{\sigma}_{i} \accentset{\circ}{\sigma}_{i+1} \accentset{\circ}{\sigma}_{i}=\accentset{\circ}{\sigma}_{i+1} \accentset{\circ}{\sigma}_{i} \accentset{\circ}{\sigma}_{i+1} & \text { for } i \in\{1, \ldots, n-2\} \\ \accentset{\circ}{\sigma}_{i} \accentset{\circ}{\sigma}_{j}=\accentset{\circ}{\sigma}_{j} \accentset{\circ}{\sigma}_{i} & \text { for } i, j \in\{1, \ldots, n-1\} \text { such that }|i-j|>1\end{array}$$
 The first relation is now ``stochastic'' in the sense that the coefficients sum to $1$ and are non--negative for $q \in [0,1]$.
 
By \eqref{eq:Rhecke}, the element
\begin{equation}\label{eq:Shecke}
\accentset{\circ}{R}_i(u) := \frac{-q(1-u)}{q^2-u} \check{R}_{i}(u) = \frac{-q(1-u)}{q^2-u} \sigma_i+\frac{q^2-1}{q^2-u} =  \frac{1-u}{q^2-u}\accentset{\circ}{\sigma}_i+\frac{q^2-1}{q^2-u} 
\end{equation}
satisfies the Yang--Baxter equation. Note that the coefficients sum to $1$. 

Define the stochastic symmetrizer by
\begin{equation}\label{def-S}
\accentset{\circ}{S}_{[i,j]}:=\prod_{i\leq a\leq j-1}^{\longrightarrow}\accentset{\circ}{R}_a(q^{-2(a-i+1)})\ \accentset{\circ}{R}_{a-1}(q^{-2(a-i)})\ \dots \accentset{\circ}{R}_i(q^{-2}) \ ,
\end{equation} 
so for example
$$
\accentset{\circ}{S}_{[i,i]}=1, \quad \accentset{\circ}{S}_{[i,i+1]} =   \frac{q^{-1}}{q+q^{-1}}\accentset{\circ}{\sigma}_i+\frac{q}{q+q^{-1}}  
$$
and
\begin{align*}
\accentset{\circ}{S}_{[i,i+2]} &= \accentset{\circ}{R}_i(q^{-2})\accentset{\circ}{R}_{i+1}(q^{-4}) \accentset{\circ}{R}_i(q^{-2}) \\
&= \frac{q^{-2}(1-q^{-4})}{(q+q^{-1})^2(q^2-q^{-4})}\left( \accentset{\circ}{\sigma}_i \accentset{\circ}{\sigma}_{i+1} \check{\sigma_i} + q^2 \accentset{\circ}{\sigma}_i \accentset{\circ}{\sigma}_{i+1}  + q^2 \accentset{\circ}{\sigma}_{i+1} \accentset{\circ}{\sigma}_{i} + q^4 \accentset{\circ}{\sigma}_i + q^4\accentset{\circ}{\sigma}_{i+1}+q^6  \right).
\end{align*}
Notice in each example that the coefficients sum to $1$. An alternative formula is
$$
\accentset{\circ}{S}_{[i,j]} = \frac{q^{(j-i+1)(j-i)}}{[j-i+1]_{q^2}!} \sum_{w} q^{-2l(w)}\accentset{\circ}{\sigma}_w
$$
where the sum is over all permutations of $\{i,i+1,\ldots,j\}$, the elements $\accentset{\circ}{\sigma}_w$ are the standard basis elements, and $l(w)$ is the length of the permutation $w$. A recursive formula is
\begin{equation}\label{recur}
\accentset{\circ}{S}_{[i,j+1]} = \frac{1}{[j-i+2]_{q^2}} \sum_{a=i}^{j+1} q^{2(a-i)} \accentset{\circ}{\sigma}_a \accentset{\circ}{\sigma}_{a+1} \cdots \accentset{\circ}{\sigma}_j \accentset{\circ}{S}_{[i,j]}.
\end{equation}
Furthermore define $\accentset{\circ}{P}^{(k,\ell)}=\accentset{\circ}{S}_{[1,k]}\accentset{\circ}{S}_{[k+1,k+\ell]}$ and $\accentset{\circ}{P}^{(\ell,k)}=\accentset{\circ}{S}_{[1,\ell]}\accentset{\circ}{S}_{[\ell+1,k+\ell]}$.

Define analogously 
\begin{equation}
 \accentset{\circ}{R}^{(k,\ell)}(u) := \accentset{\circ}{P}^{(k,\ell)} \prod_{1\leq a\leq k}^{\longleftarrow} \accentset{\circ}{R}_a\left(uq^{-2(1-a)}\right)\accentset{\circ}{R}_{a+1}\left(uq^{-2(2-a)}\right) \dots \accentset{\circ}{R}_{a+\ell-1}\left(uq^{-2(\ell-a)}\right)   \accentset{\circ}{P}^{(\ell,k)}\ ,
 \label{eq:defkls}
\end{equation}
As with the elements $\check{R}^{(k,\ell)}(u)$, the braided Yang--Baxter equation is satisfied:
$$
\accentset{\circ}{R}_{i}^{(k)}(u) \accentset{\circ}{R}_{i+1}^{(k)}(u v) \accentset{\circ}{R}_{i}^{(k)}(v)=\accentset{\circ}{R}_{i+1}^{(k)}(v) \accentset{\circ}{R}_{i}^{(k)}(u v) \accentset{\circ}{R}_{i+1}^{(k)}(v).
$$
It is a straightforward application of the braided Yang--Baxter equation satisfied by $\accentset{\circ}{R}_i(u)$ that
\begin{equation}\label{proj}
\accentset{\circ}{P}^{(k, \ell)} \accentset{\circ}{R}^{(k, \ell)}(u)=\accentset{\circ}{R}^{(k, \ell)}(u)\accentset{\circ}{P}^{(\ell, k)}
\end{equation}
Also define
\begin{equation}\label{partial}
 \accentset{\circ}{\Sigma}^{(k,\ell;p)}= \accentset{\circ}{P}^{(k,\ell)}  (\accentset{\circ}{\sigma}_k \accentset{\circ}{\sigma}_{k+1}\dots \accentset{\circ}{\sigma}_{\ell+p-1}) (\accentset{\circ}{\sigma}_{k-1} \accentset{\circ}{\sigma}_{k}\dots \accentset{\circ}{\sigma}_{\ell+p-2})  \dots  (\accentset{\circ}{\sigma}_{k-p+1} \accentset{\circ}{\sigma}_{k-p+2}\dots \accentset{\circ}{\sigma}_{\ell})  \accentset{\circ}{P}^{(\ell,k)}.
\end{equation}

Before stating the main theorem, we define some $q$--notation. Define the $q$--deformed integer by 
$$
[n]_{q^2} = \frac{1- q^{2n}}{1-q^2},
$$
the $q$--deformed factorial and binomial by 
$$
[n]_{q^2}^! = [1]_{q^2} \cdots [n]_{q^2}, \quad \quad \binom{n}{m}_{q^2} = \frac{[n]_{q^2}^!}{[m]_{q^2}^! [n-m]_{q^2}^!},
$$
and the $q$--Pochhammer symbol
$$
(a;q^2)_n = (1-a)(1-aq^2) \cdots (1-aq^{n-1}).
$$
The $q$--binomials satisfy
\begin{equation}\label{bin}
\binom{k-1}{p-1}_{q^2}  q^{2(k-p)} + \binom{k-1}{p}_{q^2} = \binom{k}{p}_{q^2}, \quad \quad
\binom{k-1}{p-1}_{q^2} + \binom{k-1}{p}_{q^2} q^{2p} = \binom{k}{p}_{q^2}
\end{equation}
The paper \cite{Po} defines weights depending on parameters $q,\mu,\nu$ by 
$$ 
\varphi(p \vert k) = \binom{k}{p}_{q^2} \mu^p \frac{ (\mu;q^2)_{k-p}  (\nu/\mu;q^2)_p }{ (\nu;q^2)_{p}  } 
$$
These weights are stochastic in the sense that
$$
\varphi(0 \vert k) + \varphi(1 \vert k ) + \cdots + \varphi(k \vert k) = 1
$$
and all $\varphi(p \vert k)$ are non--negative for the range of values $\vert q^2 \vert <1$ and $ 0 \leq \nu \leq \mu$. These weights also occur in the $q$--Hahn exclusion process \cite{veto2015,Corwin2015TheQB}.

The main theorem can now be stated.

\begin{theorem}
For $0 \leq p \leq k$ and $1 \leq k \leq \ell$, define $a_p^{(k,\ell)}(z)$ by 
\begin{align*}
a_p^{(k,l)}(z) &= \binom{k}{p}_{q^2} \frac{(q^{2l}-1)(q^{2l}-q^2)\cdots ( q^{2l}-q^{2(k-p-1)}) (1-z)(1-zq^2)\cdots (1-zq^{2(p-1)} )  }{(q^{2l}-z)(q^{2l}-zq^2)\cdots (q^{2l}-zq^{2(k-1)})}\\
&=  \binom{k}{p}_{q^2} q^{-2lp} \frac{(q^{-2l};q^2)_{k-p}(z;q^2)_p}{(zq^{-2l};q^2)_k}.
\end{align*}
In other words, $a_p^{(k,\ell)}(z) = \varphi(p \vert k)$ for $\mu = q^{-2\ell},\nu=zq^{-2\ell}$.
Then 
$$
\accentset{\circ}{R}^{(k, \ell)}(u)=\sum_{p=0}^{k} a_{p}^{(k, \ell)}(u) \accentset{\circ}{\Sigma}^{(k, \ell ; p)}.
$$
\end{theorem}

The remainder of the subsection is devoted to the proof of the Theorem.

Rather than using the closed--form expression for $a_p^{(k,\ell)}(z)$, we show that it satisfies a recursion relation and use this relation for the proof. 
\begin{lemma}
For $0 \leq p \leq k$ and $1 \leq k \leq \ell$, the coefficients $a_p^{(k,\ell)}(z)$ satisfy
\begin{equation}\label{A0}
a_0^{(1,1)}(z) = \frac{q^2-1}{q^2-z}, \quad a_1^{(1,1)}(z)= \frac{1-z}{q^2-z}
\end{equation}
and the relations
\begin{align}
&a_0^{(1,\ell)}(z) =
 a_0^{(1,\ell-1)}(z) + a_1^{(1,\ell-1)}(z) \frac{q^2-1}{q^2-z q^{2(\ell-1)}}, \quad \quad a_1^{(1,\ell)}(z) = a_1^{(1,\ell-1)}(z) \frac{ 1-zq^{2(\ell-1)} }{ q^2-zq^{2(\ell-1)} }\label{A1},\\
& a_0^{(k,\ell)}(z) = q^{2(k-1)} \frac{[\ell-k+1]_{q^2}}{[\ell]_{q^2}} a_0^{(k-1,\ell)}(zq^{-2}) a_0^{(1,\ell)}(z), \label{A2}\\
& a_{p}^{(k,\ell)}(z) = a_{p-1}^{(k-1,\ell)}(zq^{-2})\left( a_1^{(1,\ell)}(z) + \frac{[k-p]_{q^2}}{[\ell]_{q^2}}a_0^{(1,\ell)}(z) \right) + a_{p}^{(k-1,\ell)}(zq^{-2}) \cdot q^{2(k-p-1)}\frac{[\ell+p-k+1]_{q^2}}{[\ell]_{q^2}} a_0^{(1,\ell)}(z),  1 \leq p \leq k-1, \label{A3}\\
&a_k^{(k,\ell)}(z) = a_{k-1}^{(k-1,\ell)}(zq^{-2}) a_1^{(1,\ell)}(z), \label{A4}
\end{align}
\end{lemma}

\begin{proof}
Note that the equations \eqref{A0}--\eqref{A4} define $a_p^{(k,\ell)}(z)$ uniquely. So suppose that $b_p^{(k,\ell)}(z)$ also satisfy equations \eqref{A0}--\eqref{A4}. It suffices to l show that $b_p^{(k,\ell)}(z) = a_p^{(k,\ell)}(z)$. 

For $l=1$ it is immediate. By the second equation in \eqref{A1} and a telescoping product, we then have
\begin{equation}\label{M1}
a_1^{(1,\ell)}(z) = b_1^{(1,\ell)}(z).
\end{equation}
By a straightforward induction argument on $\ell$ using \eqref{A1}, it is seen that $b_0^{(1,\ell)}(z) + b_1^{(1,\ell)}(z)=1$ for all $\ell \geq 1$. Therefore
\begin{equation}\label{M2}
b_0^{(1,\ell)}(z) = 1 - b_1^{(1,\ell)}(z) = 1- a_1^{(1,\ell)}(z) = a_0^{(1,\ell)}(z).
\end{equation}

For $p=k$, we have by repeatedly applying \eqref{A4} and \eqref{M1}.
\begin{align*}
b_k^{(k,\ell)}(z) &= b_{k-1}^{(k-1,\ell)}(zq^{-2}) b_1^{(1,\ell)}(z) \\
&= b_1^{(1,\ell)}(z) b_1^{(1,\ell)}(zq^{-2})\cdots b_1^{(1,\ell)}(zq^{-2(k-1)})\\
&= \frac{(1-z)\cdots (1-zq^{2(k-1)})}{(q^{2\ell}-z)\cdots (q^{2\ell}-zq^{2(k-1)})},
\end{align*}
which equals $a_k^{(k,\ell)}(z)$.

Similarly, for $p=0$, repeatedly applying \eqref{A2} and \eqref{M2},
\begin{align*}
b_0^{(k,\ell)}(z) &= q^2 q^4 \cdots q^{2(k-1)} \frac{[\ell-k+1]_{q^2}[\ell-k+2]_{q^2} \cdots [\ell]_{q^2}}{[\ell]_{q^2} [\ell]_{q^2} \cdots [\ell]_{q^2}} \frac{q^{2\ell}-1}{q^{2\ell}-z}  \frac{q^{2\ell}-1}{q^{2\ell}-zq^2}  \cdots  \frac{q^{2\ell}-1}{q^{2\ell}-zq^{2(k-1)}} \\
&= q^2 q^4 \cdots q^{2(k-1)} \frac{(q^{2\ell}-1)(q^{2\ell-2}-1) \cdots (q^{2(\ell-k+1)}-1)}{(q^{2\ell}-z)\cdots (q^{2\ell}-zq^{2(k-1)})},
\end{align*}
which equals $a_0^{(k,\ell)}(z)$.

For $0<p<k$, plugging \eqref{M1} and \eqref{M2} into \eqref{A3}, we get
\begin{align*}
b_{p}^{(k,\ell)}(z) &= b_{p-1}^{(k-1,\ell)}(zq^2)\left( \frac{1-z}{q^{2\ell}-z}+ \frac{[k-p]_{q^2}}{[\ell]_{q^2}}\frac{q^{2\ell}-1}{q^{2\ell}-z} \right) + b_{p}^{(k-1,\ell)}(zq^2) \cdot q^{2(k-p-1)}\frac{[\ell+p-k+1]_{q^2}}{[\ell]_{q^2}} \frac{q^{2\ell}-1}{q^{2\ell}-z}\\
&= b_{p-1}^{(k-1,\ell)}(zq^2)\left( \frac{1-z}{q^{2\ell}-z}+  \frac{q^{2(k-p)}-1}{q^{2\ell}-z} \right) + b_{p}^{(k-1,\ell)}(zq^2) \cdot  \frac{q^{2\ell}-q^{2(k-p-1)}}{q^{2\ell}-z} \\
&= b_{p-1}^{(k-1,\ell)}(zq^2)\left(  \frac{q^{2(k-p)}-z}{q^{2\ell}-z} \right) + b_{p}^{(k-1,\ell)}(zq^2) \cdot  \frac{q^{2\ell}-q^{2(k-p-1)}}{q^{2\ell}-z}
\end{align*}
If we assume the result for $k-1$, then the case for $k$ amounts to showing that
$$
\binom{k-1}{p-1}_{q^2} ({ q^{2(k-p)}-z}) + \binom{k-1}{p}_{q^2}(1-zq^{2p}) = \binom{k}{p}_{q^2}(1-z).
$$
This follows from the identities \eqref{bin}. Thus $b_p^{(k,\ell)}(z) = a_p^{(k,\ell)}(z)$ for $0<p<k$.

Having show that $b_p^{(k,\ell)}(z) = a_p^{(k,\ell)}(z)$ for all values of $p,k,\ell$, the proof is now complete.
\end{proof}

We return to the proof of the theorem. The calculation is essentially the same as in \cite{crampebax}. For completeness the argument is repeated here.

Proceed by induction on the values of $k$ and $l$. For $k=l=1$ the element reduces to \eqref{eq:Shecke}.

Now suppose that the formula holds for $k=1$ and some $l-1$. By definition \eqref{eq:defkls} becomes
$$
\accentset{\circ}{R}^{(1,\ell)}(u) = \accentset{\circ}{S}_{[2,\ell+1]} \accentset{\circ}{R}_1(u) \cdots \accentset{\circ}{R}_\ell({ u q^{2(\ell-1)}}) \accentset{\circ}{S}_{[1,\ell]}.
$$
We use $\accentset{\circ}{S}_{[2,\ell+1]} = \accentset{\circ}{S}_{[2,\ell+1]} \accentset{\circ}{S}_{[2,\ell]} $ and \eqref{proj} to obtain
$$
\accentset{\circ}{R}^{(1,\ell)}(u) = \accentset{\circ}{S}_{[2,\ell+1]} \accentset{\circ}{R}_{[1,\ell]}^{(1,\ell-1)}(u) \accentset{\circ}{R}_{\ell}(u q^{2(\ell-1)}) \accentset{\circ}{S}_{[1,\ell]}.
$$
By the inductive hypothesis, 
$$
\accentset{\circ}{R}^{(1,\ell)}(u) = \accentset{\circ}{S}_{[2,\ell+1]} \left( a_1^{(1,\ell-1)}(u) \accentset{\circ}{\Sigma}^{(1,\ell-1;1)} + a_0^{(1,\ell-1)}(u) \accentset{\circ}{\Sigma}^{(1,\ell-1;0)} \right) \left(  \frac{1-u q^{2(\ell-1)} }{q^2-u q^{2(\ell-1)}}\accentset{\circ}{\sigma}_{\ell}+\frac{q^2-1}{q^2-u q^{2(\ell-1)}}  \right)\accentset{\circ}{S}_{[1,\ell]}.
$$
From the definition of $\accentset{\circ}{\Sigma}^{(1,l-1;1)}$ we obtain 

\begin{align*}
\accentset{\circ}{R}^{(1,\ell)}(u) &= \accentset{\circ}{S}_{[2,\ell+1]}\Big( a_1^{(1,l-1)}(u) \frac{1-u q^{2(\ell-1)}}{q^2-u q^{2(\ell-1)}} \accentset{\circ}{\sigma}_1 \cdots \accentset{\circ}{\sigma}_{\ell} + a_1^{(1,l-1)}(u) \frac{q^2-1}{q^2-u q^{2(\ell-1)}} \accentset{\circ}{\sigma}_1 \cdots \accentset{\circ}{\sigma}_{\ell-1} \\
& \quad \quad \quad \quad \quad\quad \quad \quad \quad \quad \quad \quad \quad  + a_0^{(1,l-1)}(u) \frac{1-u q^{2(\ell-1)}}{q^2-u q^{2(\ell-1)}}  \accentset{\circ}{\sigma}_{\ell} + a_0^{(1,l-1)}(u) \frac{q^2-1}{q^2-u q^{2(\ell-1)}}  \Big) \accentset{\circ}{S}_{[1,l]}\\
&= a_1^{(1,l-1)}(u) \frac{1-u q^{2(\ell-1)}}{q^2-u q^{2(\ell-1)}} \accentset{\circ}{\Sigma}^{(1,\ell;1)} \\
& \quad \quad \quad +\accentset{\circ}{S}_{[2,\ell+1]} \left( a_1^{(1,l-1)}(u) \frac{q^2-1}{q^2-u q^{2(\ell-1)}} + a_0^{(1,l-1)}(u) \frac{1-u q^{2(\ell-1)}}{q^2-u q^{2(\ell-1)}}   + a_0^{(1,l-1)}(u) \frac{q^2-1}{q^2-u q^{2(\ell-1)}}   \right) \accentset{\circ}{S}_{[1,\ell]}.
\end{align*}
To obtain the previous equality we used $\accentset{\circ}{S}_{[2,\ell+1]} \accentset{\circ}{\sigma}_l = \accentset{\circ}{S}_{[2,\ell+1]} $ and $\accentset{\circ}{\sigma}_j \accentset{\circ}{S}_{[1,l]} = \accentset{\circ}{S}_{[1,l]}$. By \eqref{A1} we obtain 
$$
\accentset{\circ}{R}^{(1,\ell)}(u)  = a_1^{(1,\ell-1)}(u)\accentset{\circ}{\Sigma}^{(1,\ell;1)} +  a_0^{(1,\ell-1)}(u)\accentset{\circ}{\Sigma}^{(1,\ell;0)}.
$$
This therefore proves the lemma for $k=1$ and all $ l\geq 1$. 

Now suppose that $1 < k \leq l$, and that the lemma is true for $k-1$. By definition, one has 
\begin{align*}
\accentset{\circ}{R}^{(k, \ell)}(u)=\accentset{\circ}{S}_{[1, k]} \accentset{\circ}{S}_{[k+1, k+\ell]} &\quad \accentset{\circ}{R}_{k}\left(u q^{2(1-k)}\right) \accentset{\circ}{R}_{k+1}\left(u q^{2(2-k)}\right) \ldots \accentset{\circ}{R}_{k+\ell-1}\left(u q^{2(\ell-k)}\right)\\
&\times \cdots \\
& \times \accentset{\circ}{R}_{2}\left(u q^{-2}\right) \accentset{\circ}{R}_{3}(u) \ldots \tilde{S}_{\ell+1}\left(u q^{2(\ell-2)}\right)\\
& \times \accentset{\circ}{R}_{1}(u) \accentset{\circ}{R}_{2}\left(u q^{2}\right) \ldots \accentset{\circ}{R}_{\ell}\left(u q^{2(\ell-1)}\right) \quad \accentset{\circ}{S}_{[1, \ell]} \accentset{\circ}{S}_{[\ell+1, k+\ell]}
\end{align*}
We use that $\accentset{\circ}{S}_{[1,k]} = \accentset{\circ}{S}_{[1,k]}\accentset{\circ}{S}_{[2,k]}$ and \eqref{proj} to obtain
$$
\accentset{\circ}{R}^{(k, \ell)}(u)=\accentset{\circ}{S}_{[1, k]} \accentset{\circ}{R}_{[2, k+\ell]}^{(k-1, \ell)}\left(u q^{-2}\right) \accentset{\circ}{R}_{[1, \ell+1]}^{(1, \ell)}(u) \accentset{\circ}{S}_{[\ell+1, k+\ell]}
$$

By the inductive hypothesis and the previous result for $k=1$, we get
$$
\accentset{\circ}{R}^{(k, \ell)}(u)=\accentset{\circ}{S}_{[1, k]} \sum_{p=0}^{k-1} a_{p}^{(k-1, \ell)} \left(u q^{-2}\right) \accentset{\circ}{\Sigma}_{[2, k+\ell]}^{(k-1, \ell ; p)} \left(a_1^{(1,l)}(u)  \sigma_{1} \ldots \sigma_{\ell}+a_{0}^{(1, \ell)}(u)\right) \accentset{\circ}{S}_{[1, \ell]} \accentset{\circ}{S}_{[\ell+1, k+\ell]}
$$
Replace $ \accentset{\circ}{\Sigma}_{[2, k+\ell]}^{(k-1, \ell ; p)} $ with its definition. First we have:
$$
\begin{aligned} a_1^{(1,\ell)}(u) & \accentset{\circ}{S}_{[1, k]} \accentset{\circ}{\Sigma}_{[2, k+\ell]}^{(k-1, \ell ; p)} \accentset{\circ}{\sigma}_{1} \ldots \accentset{\circ}{\sigma}_{\ell} \accentset{\circ}{S}_{[1, \ell]} \accentset{\circ}{S}_{[\ell+1, k+\ell]} \\
=a_1^{(1,\ell)}(u) & \accentset{\circ}{S}_{[1, k]} \accentset{\circ}{S}_{[k+1, k+\ell]}\left(\accentset{\circ}{\sigma}_{k} \ldots \accentset{\circ}{\sigma}_{\ell+p}\right) \ldots\left(\accentset{\circ}{\sigma}_{k-p+1} \ldots \accentset{\circ}{\sigma}_{\ell+1}\right) \accentset{\circ}{S}_{[2, \ell+1]} \accentset{\circ}{\sigma}_{1} \ldots \accentset{\circ}{\sigma}_{\ell} \accentset{\circ}{S}_{[1, \ell]} \accentset{\circ}{S}_{[\ell+1, k+\ell]} \\
=a_1^{(1,\ell)}(u) & \accentset{\circ}{S}_{[1, k]} \accentset{\circ}{S}_{[k+1, k+\ell]}\left(\accentset{\circ}{\sigma}_{k} \ldots \accentset{\circ}{\sigma}_{\ell+p}\right) \ldots\left(\accentset{\circ}{\sigma}_{k-p+1} \ldots \accentset{\circ}{\sigma}_{\ell+1}\right)\left(\accentset{\circ}{\sigma}_{1} \ldots \accentset{\circ}{\sigma}_{\ell}\right) \accentset{\circ}{S}_{[1, \ell]} \accentset{\circ}{S}_{[\ell+1, k+\ell]} \\
=a_1^{(1,\ell)}(u) & \accentset{\circ}{\Sigma}^{(k, \ell ; p+1)} \end{aligned}
$$
In the second equality we use $\accentset{\circ}{S}_{[2,\ell+1]}\accentset{\circ}{\sigma}_1 \cdots \accentset{\circ}{\sigma}_l =\accentset{\circ}{\sigma}_1 \cdots \accentset{\circ}{\sigma}_l \accentset{\circ}{S}_{[1,\ell]}$ (this is \eqref{proj} when $u\rightarrow \infty$); in the last, we notice that $\accentset{\circ}{\sigma}_1,\ldots,\accentset{\circ}{\sigma}_{k-p-1}$ in the last paranthesis commute with all elements on their left and are absorbed into $\accentset{\circ}{S}_{[1,k]}$. 

Second, we have, using \eqref{recur} for $\accentset{\circ}{S}_{[2,\ell+1]}$ and $\accentset{\circ}{S}_{[2,\ell]}\accentset{\circ}{S}_{[1,\ell]} = \accentset{\circ}{S}_{[1,\ell]}$ ,
$$
\begin{aligned} & \accentset{\circ}{S}_{[1, k]} \accentset{\circ}{\Sigma}_{[2, k+\ell]}^{(k-1, \ell ; p)} \accentset{\circ}{S}_{[1, \ell]} \accentset{\circ}{S}_{[\ell+1, k+\ell]} \\=& \accentset{\circ}{S}_{[1, k]} \accentset{\circ}{S}_{[k+1, k+\ell]}\left(\accentset{\circ}{\sigma}_{k} \ldots \accentset{\circ}{\sigma}_{\ell+p}\right) \ldots\left(\accentset{\circ}{\sigma}_{k-p+1} \ldots \accentset{\circ}{\sigma}_{\ell+1}\right) \accentset{\circ}{S}_{[2, \ell+1]} \accentset{\circ}{S}_{[1, \ell]} \accentset{\circ}{S}_{[\ell+1, k+\ell]} \\=& \frac{1}{[\ell]_{q^2}} \sum_{a=2}^{\ell+1} q^{2(a-2)} \accentset{\circ}{S}_{[1, k]} \accentset{\circ}{S}_{[k+1, k+\ell]} \underbrace{\left(\accentset{\circ}{\sigma}_{k} \ldots \accentset{\circ}{\sigma}_{\ell+p}\right) \ldots\left(\accentset{\circ}{\sigma}_{k-p+1} \ldots \accentset{\circ}{\sigma}_{\ell+1}\right)}_{=: x}\left(\accentset{\circ}{\sigma}_{a} \accentset{\circ}{\sigma}_{a+1} \ldots \accentset{\circ}{\sigma}_{\ell}\right) \accentset{\circ}{S}_{[1, \ell]} \accentset{\circ}{S}_{[\ell+1, k+\ell]} \end{aligned}
$$
One checks that $x$ commutes with $\accentset{\circ}{\sigma}_2,\ldots,\accentset{\circ}{\sigma}_{k-p-1}$, while $x \sigma_b = \sigma_{b+p}x$ if $b>x-p$. Thus we have
$$
\accentset{\circ}{S}_{[1,k]}\accentset{\circ}{S}_{[k+1,\ell+k]} x \accentset{\circ}{\sigma}_b = \accentset{\circ}{S}_{[1,k]}\accentset{\circ}{S}_{[k+1,\ell+k]} x  \text{ for any } b \neq k-p.
$$
Moreover, moving the rightmost element in each parenthesised factor of $x$ and multiplying it to $\accentset{\circ}{S}_{[\ell+1,l+\ell]}$, one finds that
$$
x\accentset{\circ}{S}_{[\ell+1,l+\ell]} = (\accentset{\circ}{\sigma}_k \cdots \accentset{\circ}{\sigma}_{\ell+p-1}) \cdots (\accentset{\circ}{\sigma}_{k-p+1} \cdots \accentset{\circ}{\sigma}_{\ell}) \accentset{\circ}{S}_{[\ell+1,l+\ell]}.
$$
Therefore
\begin{align*}
&\accentset{\circ}{S}_{[1, k]} \accentset{\circ}{\Sigma}_{[2, k+\ell]}^{(k-1, \ell ; p)} \accentset{\circ}{S}_{[1, \ell]} \accentset{\circ}{S}_{[\ell+1, k+\ell]} \\
&=\frac{1}{[\ell]_{q^2}} \sum_{a=2}^{k-p} q^{2(a-2)}  \accentset{\circ}{S}_{[1, k]} \accentset{\circ}{S}_{[k+1, k+\ell]} {\left(\accentset{\circ}{\sigma}_{k} \ldots \accentset{\circ}{\sigma}_{\ell+p}\right) \ldots\left(\accentset{\circ}{\sigma}_{k-p+1} \ldots \accentset{\circ}{\sigma}_{\ell+1}\right)}\left(\accentset{\circ}{\sigma}_{k-p} \accentset{\circ}{\sigma}_{k-p+1} \ldots \accentset{\circ}{\sigma}_{\ell}\right) \accentset{\circ}{S}_{[1, \ell]} \accentset{\circ}{S}_{[\ell+1, k+\ell]}\\
& \quad \quad + \frac{1}{[\ell]_{q^2}} \sum_{a=k-p+1}^{\ell+1} q^{2(a-2)} \accentset{\circ}{S}_{[1,k]}\accentset{\circ}{S}_{[k+1,k+\ell]} (\accentset{\circ}{\sigma}_k \cdots \accentset{\circ}{\sigma}_{\ell+p-1}) \cdots ( \accentset{\circ}{\sigma}_{k-p+1} \cdots \accentset{\circ}{\sigma}_{\ell}) \accentset{\circ}{S}_{[1,\ell]}\accentset{\circ}{S}_{[\ell+1,k+\ell]}\\
&=\frac{[k-p-1]_{q^2}}{[\ell]_{q^2}} \accentset{\circ}{\Sigma}^{(k,\ell;p+1)} + q^{2(k-p-1)} \frac{[\ell-k+p+1]_{q^2}}{[\ell]_{q^2}}\accentset{\circ}{\Sigma}^{(k,\ell;p)}.
\end{align*}
Therefore
\begin{align*}
&\accentset{\circ}{R}^{(k, \ell)}(u) \\
&=  \sum_{p=0}^{k-1} a_{p}^{(k-1, \ell)}(uq^{-2}) \left(a_1^{(1,\ell)}(u)  \accentset{\circ}{\Sigma}^{(k, \ell ; p+1)}  + a_0^{(1,\ell)}(u) \left( \frac{[k-p-1]_{q^2}}{[\ell]_{q^2}} \accentset{\circ}{\Sigma}^{(k,\ell;p+1)} + q^{2(k-p-1)} \frac{[\ell-k+p+1]_{q^2}}{[\ell]_{q^2}}   \accentset{\circ}{\Sigma}^{(k,\ell;p)}\right)\right)\\
&=a_0^{(k-1,\ell)}(uq^{-2})a_0^{(1,\ell)}(u) q^{2(k-1)}\frac{ [\ell-k+1]_{q^2}}{[\ell]_{q^2}} \accentset{\circ}{\Sigma}^{(k,\ell;0)}\\
&\quad + \sum_{p=1}^{k-1} \left[a_{p-1}^{(k-1,\ell)}(uq^{-2}) \left(a_1^{(1,\ell)}(u) + a_0^{(1,\ell)}(u) \frac{[k-p]_{q^2}}{[\ell]_{q^2}} \right) + a_p^{(k-1,\ell})(uq^{-2}) a_0^{(1,\ell)}(u) q^{2(k-p-1)} \frac{[\ell-k+p+1]_{q^2}}{[\ell]_{q^2}}   \right]\accentset{\circ}{\Sigma}^{(k,\ell;p)}\\
& \quad + a_{k-1}^{(k-1,\ell)}(uq^{-2}) a_1^{(1,\ell)}(u) \accentset{\circ}{\Sigma}^{(k,\ell;k)}.
\end{align*}
By respectively applying \eqref{A2}, \eqref{A3}, \eqref{A4} to each of the last three lines we conclude
$$
\accentset{\circ}{R}^{(k,\ell)}(u) = \sum_{p=0}^k a_p^{(k,\ell)}(u) \accentset{\circ}{\Sigma}^{(k,l;p)},
$$
which completes the proof.

\subsection{A matrix--valued stochastic solution to YBE}
By taking a representation of the fused Hecke algebra, the action of $\accentset{\circ}{R}^{(k,\ell)}(u)$ produces a matrix--valued solution the Yang--Baxter equation. Define the stochastic matrix $\accentset{\circ}{R}$ by 
$$
\accentset{\circ}{R}\left(e_{i} \otimes e_{j}\right):=\left\{\begin{array}{ll} e_{i} \otimes e_{j} & \text { if } i=j \\ e_{j} \otimes e_{i}+\left(1-q^2\right) e_{i} \otimes e_{j} & \text { if } i<j, \quad \text { where } i, j=1, \ldots, N \\ q^2e_{j} \otimes e_{i} & \text { if } i>j\end{array}\right.
$$
We write this as
$$
\accentset{\circ}{R} = q^2\sum_{i>j} E_{j,i} \otimes E_{i,j} + \sum_{i<j} E_{j,i} \otimes E_{i,j}+(1-q^2)\sum_{i<j} E_{i,i} \otimes E_{j,j}+ \sum_i  E_{i,i} \otimes E_{i,i}
$$
As with the non--stochastic case, these matrices define a representation of the fused Hecke algebra. 


For $N=2$, the matrix entries can be calculated explicitly using a result from \cite{KuanSW}. A general $N$ formula had been previously found in \cite{BoMa16}, but has a different form than the one presented here. More specifically, the spectral parameter $z$ only occurs through the coefficients $a_p^{(k,\ell)}(z)$. When $N=2$, the representation $L_{(k)}^2$ is $(k+1)$--dimensional. It can be identified with  $\accentset{\circ}{S}_{[1,k]}((\mathbb{C}^2)^{\otimes k} )\subseteq (\mathbb{C}^2)^{\otimes k}$. A basis of $L_{(k)}^2$ is then given by the images of $e_1^{\otimes j} \otimes e_2^{k-j}$ under $\accentset{\circ}{S}_{[1,k]}$, for $0 \leq j \leq k$. The matrix $\accentset{\circ}{R}^{(k,\ell)}(u)$ thus has rows and columns indexed by pairs $(k',\ell')$ where $0 \leq k' \leq k$ and $0 \leq \ell' \leq \ell$.
\begin{prop}\label{formula}
The  $(k',\ell'; \tilde{\ell},\tilde{k})$ entry of $\accentset{\circ}{R}^{(k,l)}(z)$ is given by 
\begin{multline*}
1_{\{k'+\ell' = \tilde{\ell} + \tilde{k}\}} \sum_{k'',\ell'',p} a_p^{(k,\ell)}(z)\frac{\binom{k-p}{k'-k''}_{q^2} \binom{p}{k''} q^{2k''(k-p-k'+k'')}}{\binom{k}{k'}_{q^2}} \cdot \frac{\binom{k-p}{\ell'-\ell''}_{q^2} \binom{\ell-k+p}{\ell''}_{q^2} q^{2\ell''(k-p-\ell'+\ell'')}}{ \binom{\ell}{\ell'}_{q^2}} \\
\times  J_{\ell-k+p}(k'',\ell'',\tilde{k}-k'+k'',\tilde{\ell}-\ell'+\ell'').
\end{multline*}
where
$$
J_m(k'',\ell'';\ell''+\delta,k''-\delta)=\binom{k''}{ k''-\delta}_{q^2}\left(q^{2\left(m-\ell''\right)} ; q^{-2}\right)_{\delta} q^{2\left(m-\ell''-\delta\right) (k''-\delta)}.
$$
\end{prop}
\begin{proof}
The fact that the only nonzero entries occur when $ k'+\ell' = \tilde{\ell} + \tilde{k}$ is due to the well--known weight preserving property of $R$--matrices. 

By the theorem, it suffices to show that each $\accentset{\circ}{\Sigma}^{(k,\ell;p)}$ has entries
\begin{multline*}
\sum_{k'',\ell''} \frac{\binom{k-p}{k'-k''}_{q^2} \binom{p}{k''} q^{2k''(k-p-k'+k'')}}{\binom{k}{k'}_{q^2}} \cdot \frac{\binom{k-p}{\ell'-\ell''}_{q^2} \binom{\ell-k+p}{\ell''}_{q^2} q^{2\ell''(k-p-\ell'+\ell'')}}{ \binom{\ell}{\ell'}_{q^2}}\\
\times  J_{\ell-k+p}(k'',\ell'',\tilde{k}-k'+k'',\tilde{\ell}-\ell'+\ell'')
\end{multline*}
By definition \eqref{partial},
$$
 \accentset{\circ}{\Sigma}^{(k,\ell;p)}= \accentset{\circ}{P}^{(k,\ell)}  (\accentset{\circ}{\sigma}_k \accentset{\circ}{\sigma}_{k+1}\dots \accentset{\circ}{\sigma}_{\ell+p-1}) (\accentset{\circ}{\sigma}_{k-1} \accentset{\circ}{\sigma}_{k}\dots \accentset{\circ}{\sigma}_{\ell+p-2})  \dots  (\accentset{\circ}{\sigma}_{k-p+1} \accentset{\circ}{\sigma}_{k-p+2}\dots \accentset{\circ}{\sigma}_{\ell})  \accentset{\circ}{P}^{(\ell,k)}.
$$
Let $A$ denote the product of the $\accentset{\circ}{\sigma}$ elements in the middle of the expression. We give a probabilistic interpretation of the stochastic matrices $\accentset{\circ}{P}^{(k,\ell)},A, \accentset{\circ}{P}^{(\ell,k)}$.  First, consider the map $\accentset{\circ}{P}^{(k,\ell)} = \accentset{\circ}{S}_{[1,k]} \accentset{\circ}{S}_{[k+1,k+\ell]}$. The symmetrizer $\accentset{\circ}{S}_{[1,k]} $ takes $k'$ particles and randomly places them so that $k'-k''$ particles are in the left $k-p$ positions and $k''$ are in the right $p$ positions. Similarly, the symmetrizer $\accentset{\circ}{S}_{[k+1,k+\ell]} $ takes $\ell'$ particles and randomly places them so that $\ell''$ particles are in the left $\ell-k+p$ positions and $\ell''$ are in the right $k-p$ positions. See the figure below:
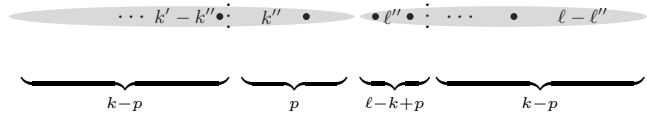
\begin{center}
 \begin{tikzpicture}[scale=0.23]

\node at (7,4) {$\dots$};
\node at (6.5,0) {$\underbrace{\phantom{blank text here must  } }_{k-p}$};
\node at (16.3,0) {$\underbrace{\phantom{blank wha } }_{p}$};
\node at (22.1,0) {$\underbrace{\phantom{blackk}}_{\ell-k+p}$};
\node at (30.5,0) {$\underbrace{\phantom{blank text here must  } }_{k-p}$};
\node at (10,4)[scale=0.8] {$k'-k''$};
\fill (12,4) circle (0.2cm);
\node at (12.5,4.5) {$\vdots$};
\node at (15,4)[scale=0.8] {$k''$};
\fill (17,4) circle (0.2cm);

\node at (22,4)[scale=0.8] {$\ell''$};
\fill (21,4) circle (0.2cm);
\fill (23,4) circle (0.2cm);
\node at (24,4.5) {$\vdots$};
\node at (26,4) {$\dots$};
\fill (29,4) circle (0.2cm);
\node at (33,4)[scale=0.8] {$\ell-\ell''$};

\fill[fill=gray,opacity=0.3] (9.8,4) ellipse (10cm and 0.6cm);
\fill[fill=gray,opacity=0.3] (28.4,4) ellipse (8.3cm and 0.6cm);

\end{tikzpicture}
\end{center}
The probabilities of how the particles are placed are given by products of the $q$--binomials. The matrix $A$ then randomly rearranges the particles in between the dotted lines. It is shown in \cite{KuanSW} that the ending locations of the particles only depend on $k''$ and $\ell''$, and not on the initial locations of the particles:
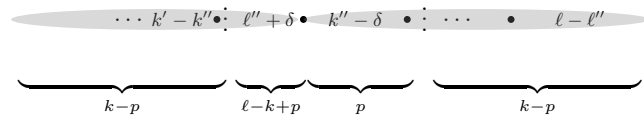
\begin{center}
 \begin{tikzpicture}[scale=0.23]

\node at (7,4) {$\dots$};
\node at (6.5,0) {$\underbrace{\phantom{blank text here must  } }_{k-p}$};
\node at (15.1,0) {$\underbrace{\phantom{blackk}}_{\ell-k+p}$};
\node at (20.3,0) {$\underbrace{\phantom{blank wha } }_{p}$};
\node at (30.5,0) {$\underbrace{\phantom{blank text here must  } }_{k-p}$};
\node at (10,4)[scale=0.8] {$k'-k''$};
\fill (12,4) circle (0.2cm);
\node at (12.5,4.5) {$\vdots$};
\node at (15,4)[scale=0.8] {$\ell''+\delta$};
\fill (17,4) circle (0.2cm);

\node at (20,4)[scale=0.8] {$k''-\delta$};
\fill (23,4) circle (0.2cm);
\node at (24,4.5) {$\vdots$};
\node at (26,4) {$\dots$};
\fill (29,4) circle (0.2cm);
\node at (33,4)[scale=0.8] {$\ell-\ell''$};

\fill[fill=gray,opacity=0.3] (8.4,4) ellipse (8.3cm and 0.6cm);
\fill[fill=gray,opacity=0.3] (27,4) ellipse (10cm and 0.6cm);

\end{tikzpicture}
\end{center} 
In order for the $\accentset{\circ}{P}^{(\ell,k)}$ term to be nonzero, we only consider ending configurations such that $k'-k'' + \ell''+\delta = \tilde{\ell} $ and $k''-\delta + \ell-\ell'' = \tilde{k}$. The sum over all such configurations is given by the $J$ term, and was calculated in \cite{KuanSW}.

We now formalize the argument above. Each matrix $\accentset{\circ}{P}^{(k,\ell)},A, \accentset{\circ}{P}^{(\ell,k)}$ is has $2^{k+\ell}$ rows and $2^{k+\ell}$ columns. The rows and columns are indexed by pairs $(\mathcal{S}_k, \mathcal{S}_\ell)$ where $\mathcal{S}_k \subseteq \{1,\ldots,k\}$ and $\mathcal{S}_\ell \subseteq \{1,\ldots ,\ell\}$, where each basis vector $e_{i_1} \otimes \cdots e_{i_k}$ corresponds to the subset of indices $j$ such that $i_j=1$. For example, the basis vector $e_1 \otimes e_2 \otimes e_2 \otimes e_1$ corresponds to the subset $\{1,4\} \subset \{1,2,3,4\}$. Therefore the $(k',\ell'; \tilde{\ell} ,  \tilde{k})$ entry of $ \accentset{\circ}{\Sigma}^{(k,\ell;p)}$ equals
$$
\sum_{\mathcal{S}_k, \tilde{\mathcal{S}}_k, \mathcal{S}_{\ell}, \tilde{\mathcal{S}}_{\ell}} \accentset{\circ}{P}^{(k,\ell)}(k',\ell'; \mathcal{S}_k, \mathcal{S}_\ell) A( \mathcal{S}_k, \mathcal{S}_\ell;  \tilde{\mathcal{S}}_\ell, \tilde{\mathcal{S}}_k) \accentset{\circ}{P}^{(\ell,k)}(   \tilde{\mathcal{S}}_\ell, \tilde{\mathcal{S}}_k, \tilde{\ell}, \tilde{k})
$$

As explained in \cite{KuanSW}, the entries of $A$ only depend on the magnitudes of $\mathcal{S}_k\cap \{1,\ldots,k-p\}, \mathcal{S}_l \cap \{1,\ldots,\ell-k+p\}, \tilde{\mathcal{S}}_k \cap \{1,\ldots,p\}, \tilde{\mathcal{S}}_{\ell} \cap \{\ell-k+p+1,\ldots,\ell\}$. So thus the entry is given by
$$
\sum_{k'',\ell'',\delta}  \sum_{\vert \mathcal{S}_k \vert= k''} \sum_{ \vert \mathcal{S}_{\ell} \vert = \ell''}  \sum_{ \vert \tilde{\mathcal{S}}_{\ell} \vert = \ell''+\delta}   \sum_{ \vert \tilde{\mathcal{S}}_{k} \vert = k''-\delta}  \accentset{\circ}{P}^{(k,\ell)}(k',\ell'; \mathcal{S}_k, \mathcal{S}_\ell) A( \mathcal{S}_k, \mathcal{S}_\ell;  \tilde{\mathcal{S}}_\ell, \tilde{\mathcal{S}}_k) \accentset{\circ}{P}^{(\ell,k)}(   \tilde{\mathcal{S}}_\ell, \tilde{\mathcal{S}}_k, \tilde{\ell}, \tilde{k}).
$$
The sum over $\mathcal{S}_k$ and $\mathcal{S}_{\ell}$ can be evaluated using the $q$--binomial theorem and  the $q$--Chu--Vandermonde identity:
$$
\binom{m+n}{k}_{q^2} = \sum_j \binom{n}{j}_{q^2} \binom{m}{k-j}_{q^2} q^{2j(m-k+j)}.
$$ 
We obtain 
\begin{multline*}
 \sum_{\vert \mathcal{S}_k \vert= k''} \sum_{ \vert \mathcal{S}_{\ell} \vert = \ell''}  \accentset{\circ}{P}^{(k,\ell)}(k',\ell'; \mathcal{S}_k, \mathcal{S}_\ell)A( \mathcal{S}_k, \mathcal{S}_\ell;  \tilde{\mathcal{S}}_\ell, \tilde{\mathcal{S}}_k)  \\
 = A( k'',\ell'';  \tilde{\mathcal{S}}_\ell, \tilde{\mathcal{S}}_k) \frac{\binom{k-p}{k'-k''}_{q^2} \binom{p}{k''} q^{2k''(k-p-k'+k'')}}{\binom{k}{k'}_{q^2}} \cdot \frac{\binom{k-p}{\ell'-\ell''}_{q^2} \binom{\ell-k+p}{\ell''}_{q^2} q^{2\ell''(k-p-\ell'+\ell'')}}{ \binom{\ell}{\ell'}_{q^2}},
\end{multline*}
where we used the fact that $A( \mathcal{S}_k, \mathcal{S}_\ell;  \tilde{\mathcal{S}}_\ell, \tilde{\mathcal{S}}_k) $ only depends on $\mathcal{S}_k, \mathcal{S}_{\ell}$ through the variables $k'',\ell''$. Thus, it remains to show that
$$
\sum_{\delta}   \sum_{ \vert \tilde{\mathcal{S}}_{\ell} \vert = \ell''+\delta}   \sum_{ \vert \tilde{\mathcal{S}}_{k} \vert = k''-\delta} A(k'',\ell'';  \tilde{\mathcal{S}}_\ell, \tilde{\mathcal{S}}_k) \accentset{\circ}{P}^{(\ell,k)}(   \tilde{\mathcal{S}}_\ell, \tilde{\mathcal{S}}_k, \tilde{\ell}, \tilde{k})= J_{\ell-k+p}(k'',\ell'',\tilde{k}-k'+k'',\tilde{\ell}-\ell'+\ell'').
$$
In order for the $\accentset{\circ}{P}^{(\ell,k)}$ term to be nonzero, we require  $\vert \tilde{S}_{\ell}\vert + k'-k'' = \ell''+\delta + k'-k''= \tilde{\ell}$ and $\vert \tilde{S}_k \vert + \ell'-\ell'' = k''-\delta + \ell'-\ell'' = \tilde{k}$, so that  $\delta = \tilde{k} - k'+k'' - \ell'' = \ell'-\ell''+k''- \tilde{\ell}$. This sum on the left--hand--side was evaluated in \cite{KuanSW}, and equals
$$
J_{\ell-k+p}( k'', \ell'', {\ell}''+\delta, {k}''-\delta).
$$
Plugging in the value of $\delta$ completes the proof.

\end{proof}

\begin{example}
Set $k=\ell=2$ and $N=2$. Then
$$
\accentset{\circ}{R}^{(2,2)}(z) =  \frac{(q^4-q^2)(q^4-1)}{(q^4-z)(q^4-zq^2)} \accentset{\circ}{\Sigma}^{(2,2;0)} + \frac{(1+q^2)(1-q^4)(z-1)}{  (q^4-z)(q^4-zq^2)} \accentset{\circ}{\Sigma}^{(2,2;1)} + \frac{(1-z)(1-zq^2)}{(q^4-z)(q^4-zq^2)} \accentset{\circ}{\Sigma}^{(2,2;2)},
$$
The representations are that $\accentset{\circ}{\Sigma}^{(2,2;0)}$ is the $9\times 9$ identity matrix, 
$$
\accentset{\circ}{\Sigma}^{(2,2;1)} = 
\left(\begin{array}{ccccccccc}1&0&0&0&0&0&0&0&0\\0&\frac{-q^4+q^2+1}{q^2+1}&0&\frac{q^4}{q^2+1}&0&0&0&0&0\\0&0&1-q^2&0&q^2&0&0&0&0\\0&\frac{1}{q^2+1}&0&\frac{q^2}{q^2+1}&0&0&0&0&0\\0&0&\frac{1}{\left(q^2+1\right)^2}&0&1-\frac{q^6+1}{\left(q^2+1\right)^2}&0&\frac{q^6}{\left(q^2+1\right)^2}&0&0\\0&0&0&0&0&\frac{-q^4+q^2+1}{q^2+1}&0&\frac{q^4}{q^2+1}&0\\0&0&0&0&1&0&0&0&0\\0&0&0&0&0&\frac{1}{q^2+1}&0&\frac{q^2}{q^2+1}&0\\0&0&0&0&0&0&0&0&1\\\end{array}\right)
$$
and
$$
 \accentset{\circ}{\Sigma}^{(2,2;2)}= 
 \left(\begin{array}{ccccccccc}1&0&0&0&0&0&0&0&0\\0&1-q^4&0&q^4&0&0&0&0&0\\0&0&\left(q^4-q^2-1\right)q^2+1&0&\left(1-q^2\right)\left(q^3+q\right)^2&0&q^8&0&0\\0&1&0&0&0&0&0&0&0\\0&0&1-q^2&0&q^2&0&0&0&0\\0&0&0&0&0&1-q^4&0&q^4&0\\0&0&1&0&0&0&0&0&0\\0&0&0&0&0&1&0&0&0\\0&0&0&0&0&0&0&0&1\\\end{array}\right)
$$
so thus $\accentset{\circ}{R}^{(2,2)}(z)  $ equals
$$
\left(
\begin{array}{ccccccccc}
 1 & 0 & 0 & 0 & 0 & 0 & 0 & 0 & 0 \\
 0 & \frac{\left(q^4-1\right) z}{q^4-z} & 0 & -\frac{q^4 (z-1)}{q^4-z} & 0 & 0 & 0 &
   0 & 0 \\
 0 & 0 & \frac{\left(q^2-1\right)^2 \left(q^2+1\right) z^2}{\left(q^2-z\right)
   \left(q^4-z\right)} & 0 & -\frac{\left(q^2-1\right) \left(q^3+q\right)^2 (z-1)
   z}{\left(q^2-z\right) \left(q^4-z\right)} & 0 & \frac{q^6 (z-1) \left(q^2
   z-1\right)}{\left(q^2-z\right) \left(q^4-z\right)} & 0 & 0 \\
 0 & \frac{z-1}{z-q^4} & 0 & \frac{q^4-1}{q^4-z} & 0 & 0 & 0 & 0 & 0 \\
 0 & 0 & -\frac{\left(q^2-1\right) (z-1) z}{\left(q^2-z\right) \left(q^4-z\right)} &
   0 & \frac{q^6 z+q^4 (1-2 z)+q^2 (z-2) z+z}{\left(q^2-z\right) \left(q^4-z\right)}
   & 0 & -\frac{q^4 \left(q^2-1\right) (z-1)}{\left(q^2-z\right) \left(q^4-z\right)}
   & 0 & 0 \\
 0 & 0 & 0 & 0 & 0 & \frac{\left(q^4-1\right) z}{q^4-z} & 0 & -\frac{q^4
   (z-1)}{q^4-z} & 0 \\
 0 & 0 & \frac{(z-1) \left(q^2 z-1\right)}{q^2 \left(q^2-z\right) \left(q^4-z\right)}
   & 0 & \frac{\left(q^2+1\right) \left(1-q^4\right) (z-1)}{\left(q^4-z\right)
   \left(q^4-q^2 z\right)} & 0 & \frac{\left(q^2-1\right)^2
   \left(q^2+1\right)}{\left(q^2-z\right) \left(q^4-z\right)} & 0 & 0 \\
 0 & 0 & 0 & 0 & 0 & \frac{z-1}{z-q^4} & 0 & \frac{q^4-1}{q^4-z} & 0 \\
 0 & 0 & 0 & 0 & 0 & 0 & 0 & 0 & 1 \\
\end{array}
\right).
$$
One can check that this is the same matrix obtained in the appendix of \cite{KuanCMP}. The $N=2$ formula goes back to  \cite{Bor16}, which is based on the non--stochastic $R$--matrix of \cite{Man14}; see also \cite{CorPetCMP,BP}. The general $N$ formula was written in \cite{BoMa16}; see also \cite{KMMO}. The formula given in Proposition \ref{formula} will always match the previous formulas when $N=2$, because $\accentset{\circ}{R}^{(k,\ell)}(z)$ satisfies the Yang--Baxter equation, commutes with the quantum group action, and is stochastic. These three formulas define the matrix uniquely.
\end{example}

\bibliographystyle{alpha}
\bibliography{Baxterization}

\end{document}